\documentclass[reqno, 12pt]{amsart}
\usepackage{amsmath}
\usepackage{amssymb}
\usepackage{amsthm}
\usepackage{mathtools}
\usepackage{tikz-cd}
\usepackage{enumerate}
\usepackage[mathscr]{eucal}
\usepackage{graphicx}

\newtheorem{theorem}{Theorem}[section]
\newtheorem{lemma}[theorem]{Lemma}
\newtheorem{proposition}[theorem]{Proposition}
\newtheorem{corollary}[theorem]{Corollary}
\newtheorem{definition}[theorem]{Definition}

\theoremstyle{definition}
\newtheorem{example}[theorem]{Example}
\newtheorem{remark}[theorem]{Remark}

\begin{document}
	\title[Reflexive extended locally convex spaces]{Reflexive extended locally convex spaces}
	\author{Akshay Kumar \and Varun Jindal}
		\address{Akshay Kumar: Department of Mathematics, Malaviya National Institute of Technology Jaipur, Jaipur-302017, Rajasthan, India}
	\email{akshayjkm01@gmail.com}
	
	\address{Varun Jindal: Department of Mathematics, Malaviya National Institute of Technology Jaipur, Jaipur-302017, Rajasthan, India}
	\email{vjindal.maths@mnit.ac.in}

\maketitle

\begin{abstract}  
For an extended locally convex space (elcs) $(X,\tau)$, the authors in \cite{doelcs} studied the topology $\tau_{ucb}$ of uniform convergence on bounded subsets of $(X,\tau)$ on the dual $X^*$ of $(X,\tau)$. In the present paper, we use the topology $\tau_{ucb}$ to explore the reflexive property of extended locally convex spaces. It is shown that an elcs is (semi) reflexive if and only if any of its open subspaces is (semi) reflexive. For an extended normed space, we show that reflexivity is a three-space property.    		
	
 \end{abstract}	

\section{Introduction}
In classical functional analysis, our attention is directed towards the study of locally convex spaces. An important characterization of a locally convex space is that a collection of seminorms induces its topology. However, in various problems, we encounter functions that possess all the properties of a seminorm (or even a norm) but can also assume infinite value.      

An extended norm on a vector space $X$ is a function satisfying all the properties of a norm and, in addition, can also attain infinite value. A vector space $X$ together with an extended norm is called an extended normed linear space (enls). These spaces were first formally studied by Beer in \cite{nwiv} and further developed by Beer and Vanderwerff in \cite{socsiens, spoens}. 

Salas and Tapia-Garc{\'\i}a introduced the concept of an extended locally convex space in \cite{esaetvs}, which is a generalization of an extended normed linear space. These new extended spaces are different from the classical locally convex spaces as the scalar multiplication in these spaces may not be jointly continuous. As a result, the conventional theory of locally convex spaces may not be directly applicable to these new spaces. To address this problem, the finest locally convex topology (flc topology) for an (elcs) $(X, \tau)$ which is coarser than $\tau$ was studied in \cite{flctopology}. It was shown in \cite{flctopology} that if $\tau_F$ is the flc topology for an elcs $(X, \tau)$, then both $(X, \tau)$ and $(X, \tau_F)$ have the same dual $X^*$ (the collection of all continuous linear functionals). 

In \cite{doelcs}, the authors employed the flc topology to examine the dual of an elcs. Specifically, they studied the weak topology on an elcs $(X, \tau)$ and the weak$^*$ topology on the dual $X^*$ of $(X, \tau)$. Besides this, on $X^*$, they also studied the topology $\tau_{ucb}$ of uniform convergence on bounded subsets of $(X, \tau)$. 

In the present paper, we use the topology $\tau_{ucb}$ to study reflexive extended locally convex spaces. 

The paper is organized as follows: the second section presents all  essential preliminary results and definitions. In Section 3, we define and study reflexive extended locally convex spaces. More specifically, we relate the reflexivity of an elcs $(X, \tau)$ with the reflexivity of its finest space $(X, \tau_F)$, where $\tau_F$ is the corresponding flc topology. We also show that an elcs $(X, \tau)$ is reflexive if and only if any of its open subspaces is reflexive. Further, in the case of an enls, we prove that the reflexivity property is a three-space property.  As an application of our results, in the final section, we look at the reflexivity of some well known function spaces. 


\section{Preliminaries}

The underlying field of a vector space is denoted by $\mathbb{K}$ which is either $\mathbb{R}$ or $\mathbb{C}$. We adopt the following conventions for $\infty$: $\infty.0=0.\infty=0$; $\infty+\alpha=\alpha+\infty=\infty$ for every $\alpha\in\mathbb{R}$; $\infty.\alpha=\alpha.\infty=\infty$ for $\alpha>0$; $\inf\{\emptyset\}=\infty$.  

An \textit{extended seminorm} $\rho:X\to[0, \infty]$ on a vector space $X$ is a function which satisfies the following properties.
\begin{itemize}
	\item[(1)] $\rho(\alpha x)=|\alpha|\rho(x)$ for each $x\in X$ and scalar $\alpha$;
	\item[(2)] $\rho(x+y)\leq \rho(x)+\rho(y)$ for all $x,y\in X$.\end{itemize}

An \textit{extended norm} $\parallel\cdot\parallel:X\to[0, \infty]$ is an extended seminorm with the property: if $\parallel x\parallel=0$, then $x=0$. A vector space $X$ endowed with an extended norm $\parallel\cdot\parallel$ is called an \textit{extended normed linear space} (or \textit{extended normed space}) (enls, for short), and it is denoted by $(X, \parallel\cdot\parallel)$. The \textit{finite subspace} of an enls $(X, \parallel\cdot\parallel)$ is defined as $$X_{fin}=\{x\in X:\parallel x\parallel<\infty\}.$$ 
Note that the extended norm $\parallel\cdot\parallel$ on $X_{fin}$ is actually a norm. Therefore $(X_{fin}, \parallel\cdot\parallel)$ is a conventional normed linear space. 

We say an enls $(X, \parallel\cdot\parallel)$ is an \textit{extended Banach space} if it is complete with respect to the metric $d(x, y)= \min\{\parallel x-y\parallel, 1\}$ for all $x, y\in X$. One can prove that an enls $(X, \parallel\cdot\parallel)$ is an extended Banach space if and only if the finite space $(X_{fin}, \parallel\cdot\parallel)$ is a Banach space. For details about extended normed linear spaces, we refer to \cite{nwiv, socsiens, spoens}.  

Suppose $(X, \parallel\cdot\parallel_1)$ and $(Y, \parallel\cdot\parallel_2)$ are extended normed linear spaces. Then for a continuous linear map $T:X\to Y$, we define $$\parallel T\parallel_{op}=\sup\{\parallel T(x)\parallel_2: \parallel x\parallel_1\leq 1\}.$$ In particular, if $f\in X^*$, then $\parallel f\parallel_{op}=\sup\{|f(x)|:\parallel x\parallel_1\leq 1\}.$ The following points about an enls $(X, \parallel\cdot\parallel)$ are given in \cite{nwiv}.
\begin{enumerate}
\item $X_{fin}$ is open in $(X, \parallel\cdot\parallel)$. 	
\item $\parallel f\parallel_{op}=\parallel f|_{X_{fin}}\parallel_{op}$ for every $f\in X^*$, where $f|_{X_{fin}}$ is the restriction of $f$ on the normed linear space $(X_{fin}, \parallel\cdot\parallel)$.	
\item For any linear functional $f$ on $X$, we have $f\in X^*$ if and only if $f|_{X_{fin}}$ is continuous on $X_{fin}$. 
\item If $f\in X^*$ and $\parallel f\parallel_{op}\neq 0$, then $|f(x)|\leq \parallel f\parallel_{op}\parallel x\parallel$ for every $x\in X$.
\end{enumerate}

\noindent It follows from the point (2) given above that $\parallel\cdot\parallel_{op}$ may not be a norm on the dual $X^*$ of an enls $(X, \parallel\cdot\parallel)$. However, following \cite{nwiv}, we call it the \textit{operator norm} in the sequel.

A vector space $X$ endowed with a Hausdorff topology $\tau$ is said to be an \textit{extended locally convex space} (elcs, for short) if $\tau$ is induced by a collection $\mathcal{P}=\{\rho_i:i\in\mathcal{I}\}$ of extended seminorms on $X$, that is, $\tau$ is the smallest topology on $X$ under which each $\rho_i$ is continuous. We define $$X_{fin}^\rho=\{x\in X:\rho(x)<\infty\}$$ for any extended seminorm $\rho$ on $X$ and the \textit{finite subspace} $X_{fin}$ of an elcs $(X, \tau)$ by $$X_{fin}=\bigcap\left\lbrace X_{fin}^\rho : \rho  \text{ is continuous on } (X, \tau)\right\rbrace.$$    

Suppose $(X, \tau)$ is an elcs and $\tau$ is induced by a family $\mathcal{P}$ of extended seminorms on $X$. Then the following facts are either easy to verify or given in \cite{esaetvs}. 
\begin{itemize}
	\item[(1)] There exists a neighborhood base $\mathcal{B}$ at $0$ in $(X, \tau)$ such that each element of $\mathcal{B}$ is absolutely convex (balanced and convex);
	\item[(2)] $X_{fin}$ with the subspace topology is a locally convex space;
	\item[(3)] $X_{fin}=\bigcap_{\rho\in\mathcal{P}} X_{fin}^\rho;$
	\item[(4)] if $\rho$ is any continuous extended seminorm on $(X, \tau)$, then $X_{fin}^\rho$ is a clopen subspace of $(X, \tau)$;
	\item[(5)] $X_{fin}$ is an open subspace of $(X, \tau)$ if and only if there exists a continuous extended seminorm $\rho$ on $(X, \tau)$ such that $X_{fin}=X_{fin}^\rho$. In this case, we say $(X, \tau)$ is a \textit{fundamental elcs}. \end{itemize}

It is shown in Proposition 4.7 of \cite{esaetvs} that if $\mathcal{B}$ is a neighborhood base at $0$ in an elcs $(X, \tau)$ consisting of absolutely convex sets, then $\tau$ is induced by the collection $\{\mu_{U}: U\in\mathcal{B}\}$ of Minkowski functionals.

\begin{definition}\label{Minkowski funcitonal}	{\normalfont(\cite{esaetvs})}	\normalfont	Suppose $U$ is any absolutely convex set in an elcs $(X, \tau)$. Then the \textit{Minkowski functional} $\mu_U:X\rightarrow[0,\infty]$  for $U$  is defined as  $$\mu_U(x)=\inf\{\lambda>0: x\in \lambda U\}.$$
\end{definition}

The following facts are immediate from the above definition.
\begin{enumerate}
	\item The Minkowski functional $\mu_{U}$ for the set $U$ is an extended seminorm on $X$. In addition, if  $U$ is absorbing, then $\mu_U$ is a seminorm on $X$.
	\item  $\{x\in X:\mu_U(x)<1\}\subseteq U\subseteq \{x\in X:\mu_U(x)\leq 1\}$.
	\item  The Minkowski functional $\mu_U$ is  continuous on $X$ if and only if $U$ is a neighborhood of $0$ in $(X,\tau)$.
\end{enumerate} 

If $A$ is any nonempty set in a topological space $(X,\tau)$, then we denote the closure and interior of $A$ in $(X,\tau)$ by $\text{Cl}_\tau(A)$  and  $\text{int}_\tau(A)$, respectively. We also adopt the following terminology for an elcs $(X, \tau)$. 
\begin{enumerate}
\item If $U$ is any absolutely convex subset of $X$, then $X_{fin}^U=\{x\in X: \mu_U(x)<\infty\}$. 
\item If $A\subseteq X$, then ab($B$) is the smallest absolutely convex set in $X$ that contains $A$.
\item If $A\subseteq X$, then $A^\circ=\{f\in X^*: |f(x)|\leq 1 \text{ for every } x\in A\}$ is called the polar of $A$ in $X^*$.
\item If $A\subseteq X^*$, then $A_\circ=\{x\in X: |f(x)|\leq 1 \text{ for every } f\in A\}$ is called the polar of $A$ in $X$.   
\end{enumerate}  For other terms and definitions, we refer to \cite{lcsosborne, tvsschaefer, willard}.

 
\section{Reflexivity}


The aim of this section is to explore reflexive extended locally convex spaces. To define reflexivity property of an elcs $(X,\tau)$, we first need to define the topology $\tau_{ucb}$, on $X^*$, of uniform convergence on bounded subsets of $(X,\tau)$. For an elcs, the topology $\tau_{ucb}$ has been studied extensively in \cite{doelcs}. We first give the definition of a bounded set in an elcs.      

%

\begin{definition}\label{bounded set in an elcs}{\normalfont(\cite{flctopology})} \normalfont	Suppose $(X, \tau)$ is an elcs. Then $A\subseteq X$ is said to be \textit{bounded} in $(X, \tau)$ if for every neighborhood $U$ of $0$, there exist $r>0$ and a finite set $F\subseteq X$ such that $A\subseteq F+rU$. \end{definition}

\noindent The following points about bounded sets in an elcs $(X,\tau)$ are either easy to prove or given in \cite{flctopology}.
\begin{enumerate}
	\item Every finite subset of $X$ is bounded. 
	\item Every subset of a bounded set is bounded.
	\item When $(X, \tau)$ is a conventional locally convex space, $A\subseteq X$ is bounded in the sense of Definition \ref{bounded set in an elcs} if and only if it is absorbed by each neighborhood of $0$ in $(X, \tau)$.
	\item Suppose $(Y,\sigma)$ is any elcs and $T:X\rightarrow Y$ is a  continuous linear operator. Then for every bounded set $A$ in $(X,\tau)$, $T(A)$ is bounded in $(Y,\sigma)$. In particular, for every $f\in X^*$, $f(A)$ is bounded in $\mathbb{K}$. 
	\item No subspace $\left( \text{other than the zero subspace} \right)$  is bounded.
	\item If $x_n\to 0$ in $(X,\tau)$, then $\{x_n:n\in \mathbb{N}\}$ is  bounded in  $(X, \tau)$. \end{enumerate}

\begin{definition}\label{definition of topology of uniform topology}\normalfont (\cite{doelcs})
	Let $(X, \tau)$ be an elcs. Then the topology $\tau_{ucb}$, on $X^*$, of \textit{uniform convergence on bounded subsets of} $(X, \tau)$ is induced by the collection $\mathcal{P}=\{\rho_B: B ~\text{is bounded subset of}~ (X, \tau)\}$ of seminorms on $X^*$, where $\rho_B(\phi)=\sup_{x\in B}|\phi(x)| ~\text{for}~ \phi\in X^*$.\end{definition}

\noindent The following points for an elcs $(X, \tau)$ are either easy to verify or given in \cite{nwiv, doelcs}. 
\begin{enumerate}
 \item $(X^*, \tau_{ucb})$ is a locally convex space and $\mathcal{B}=\{B^\circ: B ~\text{is bounded in}~ (X, \tau)\}$ is a neighborhood base at $0$ in $(X^*, \tau_{ucb})$, where the polar $B^\circ$ of $B\subseteq X$ is given by $B^\circ=\{\phi\in X^*: |\phi(x)|\leq 1 \text{ for all } x\in B\}$.
 \item Recall from \cite{doelcs} that the \textit{weak$^*$ topology} on $X^*$ is induced by the collection $\{\rho_x: x\in X\}$ of seminorms, where $\rho_x(f)=|f(x)|$ for every $f\in X^*$. Clearly,  $\tau_{w^*}$ is coarser than $\tau_{ucb}$.
 \item If $(X, \parallel\cdot\parallel)$ is an enls such that $X=X_{fin}\oplus M$, then $(X^*, \tau_{ucb})$ is isomorphic (linear homeomorphic) to $(X_{fin}^*, \parallel\cdot\parallel_{op})\times (M^*, \tau_{w^*})$, where $\tau_{w^*}$ is the weak$^*$ topology on the dual $M^*$ of the enls $(M, \parallel\cdot\parallel)$ (see, Theorem 4.11 in \cite{nwiv}).
\end{enumerate}
%

Recall that when $(X,\tau)$ is a classical locally convex space, the topology $\tau_{ucb}$ is more popularly known as the strong topology.  In particular, for an elcs $(X, \tau)$ with the flc topology $\tau_F$ by the \textit{strong topology} $\tau_s$ on $X^*$, we mean the topology of uniform convergence on bounded subsets of the locally convex space $(X, \tau_F)$.  

\begin{remark}It is easy to prove that if $\tau_F$ is the flc topology of an elcs $(X, \tau)$, then every bounded set in $(X, \tau)$ is bounded in $(X, \tau_F)$. Converse may not  be true (see, Proposition 5.3 in \cite{doelcs}). Therefore $\tau_{ucb}$ is coarser than $\tau_s$. 
\end{remark}

For an elcs $(X,\tau)$ and $x \in X$, if a net $(f_\lambda)$ converges to $f$ in $(X^*,\tau_{ucb})$, then $f_\lambda(x) \to f(x)$. Consequently, the map $J_x: (X^*,\tau_{ucb}) \to \mathbb{K}$ defined by $J_x(f)= f(x)$ for $f \in X^*$ is a continuous linear functional. Hence the \textit{canonical map} $J: X \to (X^*,\tau_{ucb})^*$ defined by $J(x) = J_x$ for all $x \in X$ is well defined. 

\begin{remark}
For an enls $(X,\parallel\cdot\parallel)$,  the map $J:(X,\parallel\cdot\parallel)\rightarrow(X^*,\parallel\cdot\parallel_{op})^*$ may not be well defined. Since for every $z\notin X_{fin}$ and $n\in \mathbb{N}$, we can find a continuous linear functional $f_n$ on $X$ such that $f_n(z)=n$ and $f(x)=0$ for every $x\in X_{fin}$. Clearly, $\parallel f_n\parallel_{op}=0$ for every $n\in\mathbb{N}$. Consequently, $f_n\to 0$ in $(X^*, \parallel\cdot\parallel_{op})$ but $|J_z(f_n)|=n\nrightarrow 0$.
\end{remark} 
  
\begin{definition}\normalfont Let $(X, \tau)$ be an elcs. Then we say $X$ is \textit{semi-reflexive} if the canonical map $J$ is surjective, and we say $X$ is \textit{reflexive} if $J$ is both surjective and continuous when $(X^*,\tau_{ucb})^*$ is equipped with the topology of uniform convergence on bounded subsets of $(X^*,\tau_{ucb})$. \end{definition} 
\begin{example}\label{basic examples of a reflexive elcs}  Let $X$ be a vector space with the \textit{discrete extended norm} defined by \[ \parallel x\parallel_{0, \infty}=
	\begin{cases}
	\text{0;} &\quad\text{if $x=0 $}\\
	\text{$\infty$;} &\quad\text{if $x\neq 0.$ }\\
	\end{cases}\] Then only finite subsets are bounded in $(X, \parallel\cdot\parallel_{0, \infty})$. Since the given space is a discrete space, the canonical map $J$ is continuous. Now, if $\psi \in (X^*, \tau_{ucb})^*$, then there exists a finite set $A=\{x_1, x_2, ...,x_n\} \subseteq X$ such that $A^\circ \subseteq \psi^{-1}(-1,1)$. It is easy to show that $\bigcap_{x\in A} J_x^{-1}(0)\subseteq \psi^{-1}(0)$ (if $\epsilon>0$ and $f(x_j)=0$ for $1\leq j\leq n$, then $\frac{1}{\epsilon}f(x_j)=0$ for $1\leq j\leq n$. So $|\psi(f)|<\epsilon$). By Lemma 3.9, p. 67  in \cite{faaidg}, $\psi$ is a linear combination of $J_{x_1}, J_{x_2},.., J_{x_n}$. Therefore $J$ is surjective. Hence $(X, \parallel\cdot\parallel_{0, \infty})$ is reflexive. \end{example}

	\begin{proposition}\label{continuity of the canonical map}
		Let $(X, \parallel\cdot\parallel)$ be an enls. Then $J:(X,\parallel\cdot\parallel)\rightarrow (X^*, \tau_{ucb})^*$  is always continuous.  \end{proposition}
	\begin{proof}
		Let  $x_n\rightarrow 0$ in $(X, \parallel\cdot\parallel)$ and let $B$ be a bounded subset of $(X^*,\tau_{ucb})$. Then there exist  $\alpha>0$ and $n_0\in\mathbb{N}$ such that $\parallel f\parallel_{op}\leq \alpha$ and $x_n\in X_{fin}$ for each $f\in B$, $n\geq n_0$. Consequently,	$$|J_{x_n}(f)|=|f(x_n)|\leq \alpha \parallel x_n\parallel ~\text{for every } f\in B~\text{and}~ n\geq n_0.$$  Therefore $J_{x_n}\rightarrow 0$ in $(X^*, \tau_{ucb})^*$. Hence $J$ is continuous.\end{proof}
	
	\begin{corollary}\label{equivalence of reflexive and semireflexive in an enls}
	Let $(X, \parallel\cdot\parallel)$ be an enls. Then $X$ is reflexive if and only if it is semi-reflexive.  \end{corollary}
 
We next study the reflexivity of an elcs $(X, \tau)$ in relation to the properties of the corresponding finest space $(X, \tau_{F})$. 
 
Recall that a locally convex space $(X, \tau)$ is said to be \textit{barreled} if each barrel (closed, absolutely convex and absorbing set) is a neighborhood of $0$.       	
\begin{proposition}\label{reflexive of an elcs imply barrelledness of finest space}
	Let $(X, \tau)$ be a reflexive elcs with the flc topology $\tau_{F}$. Then $(X, \tau_F)$ is barreled.\end{proposition}
\begin{proof}
	 Let $B$ be a barrel in $(X, \tau_F)$. By Theorem 8.8.3, p. 251 in \cite{tvsnarici}, $B^\circ$ is pointwise bounded. Since $(X, \tau)$ is reflexive, we have  $(X^*, \tau_{ucb})^*= \{J_x: x\in X \}$. So $B^\circ$ is bounded in $(X^*, \tau_{ucb})$. Consequently, $(B^\circ)^\circ$ is a neighborhood of $0$ in $(X^*, \tau_{ucb})^*$. As $J:(X, \tau)\to (X^*, \tau_{ucb})^*$ is continuous, $J^{-1}\left( (B^\circ)^\circ\right)$ is a neighborhood of $0$ in $(X, \tau)$. Note that $J^{-1}\left((B^\circ)^\circ\right)= (B^\circ)_\circ$. Consequently, by applying Bipolar theorem in $(X,\tau_F)$, we have  $(B^\circ)_\circ = B$. Therefore $B$ is a neighborhood of $0$ in $(X, \tau)$. Since $B$ is an absorbing and absolutely convex neighborhood of $0$ in $(X, \tau)$, the Minkowski functional $\mu_{B}$ is a continuous seminorm on $(X, \tau)$. By Theorem 3.5 in \cite{flctopology}, $\mu_B$ is a continuous seminorm on $(X, \tau_F)$. Therefore $B$ is a neighborhood of $0$ in $(X, \tau_F)$. Hence $(X, \tau_F)$ is barreled.\end{proof}

\begin{theorem}\label{reflexivity of an elcs and finest space}
	Let $(X, \tau)$ be an elcs with the flc topology $\tau_{F}$. If $(X, \tau_F)$ is reflexive, then $(X, \tau)$ is reflexive. Converse holds if $(X, \tau_F)$ is semi-reflexive.\end{theorem}	
	\begin{proof} Let $(X, \tau_F)$ be reflexive. Since $ \tau_{ucb}\subseteq \tau_s$ and $(X^*, \tau_s)^*=\{J_x : x\in X\}$, we have $(X^*, \tau_{ucb})^*=\{J_x : x\in X\}$. Therefore $(X, \tau)$ is semi-reflexive. Now, let $B$ be a bounded set in $(X^*, \tau_{ucb})$. Then $B$ is bounded in $(X^*, \tau_{w^*})$. By Theorem 8.8.3, p. 241 in \cite{tvsnarici}, $B_\circ$ is an absorbing subset of $X$. Note that  $J^{-1}(B^\circ)=\{x\in X: |J_x(f)|\leq 1 ~\text{for}~ f\in B\}=B_\circ$ is a barrel in $(X, \tau_F)$ (see, Theorem 8.3.6, p. 234 in \cite{tvsnarici}). Since $(X, \tau_F)$ is reflexive, by Theorem 15.2.6, p. 490 in \cite{tvsnarici}, $(X, \tau_F)$ is barreled. Consequently, $B_\circ$ is a neighborhood of $0$ in $(X, \tau_F)$. So $J^{-1}(B^\circ)=B_\circ$ is a neighborhood of $0$ in $(X, \tau)$. Therefore the canonical map $J$ is continuous on $(X, \tau)$. Hence $(X, \tau)$ is reflexive.   
	
   Conversely, suppose $(X, \tau)$ is reflexive and $(X, \tau_F)$ is semi-reflexive. Then by Proposition \ref{reflexive of an elcs imply barrelledness of finest space}, $(X, \tau_F)$ is barreled. By Theorem 15.2.6, p. 490 in \cite{tvsnarici}, $(X, \tau_F)$ is reflexive. \end{proof}
\begin{corollary}\label{reflexivity of a fundamental elcs}
	Let $(X, \tau)$ be an elcs with the flc topology $\tau_{F}$. Suppose anyone of the following conditions holds
	\begin{itemize}
		\item[(1)] $(X^*, \tau_{ucb})$ is barreled;
		\item[(2)] $(X, \tau)$ is a fundamental elcs.\end{itemize} Then $(X, \tau)$ is reflexive if and only if $(X, \tau_F)$ is reflexive.  \end{corollary}
	\begin{proof}
		Suppose anyone of the given assumptions holds. Then by Theorem 5.9 and Theorem 5.12(2) in \cite{doelcs}, we have $\tau_{ucb}=\tau_{s}$. Therefore  $(X, \tau)$ is semi-reflexive if and only if $(X, \tau_F)$ is semi-reflexive. By Theorem \ref{reflexivity of an elcs and finest space}, $(X, \tau)$ is reflexive if and only if $(X, \tau_F)$ is reflexive.   \end{proof}
	
\begin{corollary}\label{reflexivity of an enls}Suppose $(X, \parallel\cdot\parallel)$ is an enls. Then $(X, \parallel\cdot\parallel)$ is reflexive if and only if $(X, \tau_F)$ is reflexive.\end{corollary}

Our next theorem relates the reflexivity of an elcs $(X,\tau)$ with the reflexivity of its open subspaces.
\begin{proposition} Suppose $M$ is an open subspace of an elcs $(X, \tau)$. Then there exists a continuous extended seminorm $\rho$ on $(X, \tau)$ such that $X_{fin}^\rho=M$.\end{proposition}
\begin{proof} Since $M$ is an open subspace of $(X, \tau)$, there exists a continuous extended seminorm $\mu$ on $(X, \tau)$ such that $X_{fin}^\mu\subseteq M$. Define an extended seminorm $\rho$ on $X$ by \[ \rho(x)=
	\begin{cases}
		\text{0;} &\quad\text{if $x\in M$}\\
		\text{$\infty$;} &\quad\text{if $x\notin M.$ }\\
	\end{cases}\] Note that $\rho(x)\leq \mu(x)$ for every $x\in X$. Consequently, $\rho$ is continuous on $(X, \tau)$. It is easy to see that $X_{fin}^\rho=M$.  \end{proof}

We need the following remark and lemma in the proof of Theorem \ref{reflexivity of X with finite subspaces}. 

\begin{remark}\label{reflexivity of M}
Suppose $\rho$ is a continuous extended seminorm on an elcs $(X, \tau)$ and $M$ is a subspace of $X$ with $X=X_{fin}^\rho\oplus M$. Then $(M, \tau|_M)$ is a discrete space. Therefore $\tau|_M$ is induced by the discrete extended  norm $\parallel\cdot\parallel_{0, \infty}$. If $\tau_F$ is the flc topology for $(X, \tau)$, then by Theorem 4.1 in \cite{flctopology}, the flc topology for $(M, \tau|_M)$ is $\tau_F|_M$. By Example \ref{basic examples of a reflexive elcs} and Corollary \ref{reflexivity of an enls}, both the spaces $(M, \tau|_M)$ and $(M, \tau_F|_M)$ are reflexive.\end{remark}
 
 \begin{lemma}\label{X is isomorphic to finite subspace and M}
  Let $(X, \tau)$ be an elcs and let $\rho$ be a continuous extended seminorm on $(X, \tau)$. If $M$ is a subspace of $X$ with $X=X_{fin}^\rho\oplus M$ and $\tau_F$ is the flc topology for $(X, \tau)$, then  $(X, \tau_F)$ is isomorphic to the product space $\left(X_{fin}^\rho, \tau_F|_{X_{fin}^\rho}\right)\times \left( M, \tau_F|_M \right)$.\end{lemma}
\begin{proof}
	Consider the map $\Psi: (X, \tau_F) \to \left(X_{fin}^\rho, \tau_F|_{X_{fin}^\rho}\right)\times \left( M, \tau_F|_M \right)$ defined by $\Psi(x=x_f+x_m)=(x_f, x_m)$. Then $\Psi$ is linear and bijective. Note that  if $U$ and $V$ are neighborhoods of $0$ in $\left(X_{fin}^\rho, \tau_F|_{X_{fin}^\rho}\right)\text{ and } \left( M, \tau_F|_M \right)$, respectively, then there exist continuous seminorms $\mu$ and $\nu$ on $\left(X_{fin}^\rho, \tau_F|_{X_{fin}^\rho}\right)\text{ and } \left( M, \tau_F|_M \right)$ such that $\mu^{-1}([0, 1))\subseteq U$ and $\nu^{-1}([0, 1))\subseteq V$. It is easy to see that $\lambda(x)=\mu(x_f)+\nu(x_m)$ for $x\in X$ and $x=x_f+x_m$ is a continuous seminorm on $(X, \tau)$ as $X_{fin}^\rho$ is an open subspace of $(X, \tau)$ and $\lambda=\mu$ on $X_{fin}^\rho$. By Theorem 3.5 in \cite{flctopology}, $\lambda$ is continuous on $(X, \tau_F)$. Note that $\Psi(\lambda^{-1}([0, 1)))\subseteq U\times V$. Therefore $\Psi$ is continuous. 
	
  For the continuity of  $\Psi^{-1}$, let $\rho$ be a continuous seminorm on $(X, \tau_{F})$. Then $\rho|_{X_{fin}^\rho}$ and $\rho|_M$ are continuous on $(X_{fin}^\rho, \tau_F|_{X_{fin}^\rho})$ and $(M, \tau_F|_M)$, respectively. Note that $$\Psi^{-1}\left(\rho|^{-1}_{X_{fin}^\rho}\left( \left[ 0,\frac{1}{2}\right) \right) \times \rho|^{-1}_M\left( \left[ 0,\frac{1}{2}\right) \right) \right)\subseteq \rho^{-1}([0, 1)).$$ Therefore $\Psi^{-1}$ is continuous.  \end{proof}

\begin{theorem}\label{reflexivity of X with finite subspaces}
	Suppose $(X, \tau)$ is an elcs with the flc topology $\tau_{F}$. Then the following statements are equivalent.
	\begin{itemize}
		\item[(1)] $\left(X, \tau \right)$ is reflexive;
		\item[(2)] $\left(X_{fin}^\rho, \tau|_{X_{fin}^\rho} \right)$ is reflexive, for every continuous extended seminorm $\rho$ on $X$;
		\item[(3)] $\left(X_{fin}^\rho, \tau|_{X_{fin}^\rho} \right)$ is reflexive, for some continuous extended seminorm $\rho$ on $X$.	\end{itemize}\end{theorem}
	\begin{proof}
		(1)$\Rightarrow$(2). Suppose $\rho$ is a continuous extended seminorm on $(X, \tau)$ and $\psi$ is a continuous linear functional on $\left( \left( X_{fin}^\rho\right)^*, \tau_{ucb}^\rho\right) $, where $\tau_{ucb}^\rho$ is the topology of uniform convergence on bounded subsets of $\left(X_{fin}^\rho, \tau|_{X_{fin}^\rho} \right)$. Define a linear functional $\Psi$ on $X^*$ by $\Psi(f)=\psi(f|_{X_{fin}^\rho})$ for $f\in X^*$. Suppose $(f_\alpha)$ is a net in $(X^*, \tau_{ucb})$ converging to $0$. Then $f_\alpha|_{X_{fin}^\rho}\to 0$ in $\left( \left( X_{fin}^\rho\right)^*, \tau_{ucb}^\rho\right)$ as every bounded subset of $\left(X_{fin}^\rho, \tau|_{X_{fin}^\rho} \right)$ is also bounded in $(X, \tau)$. Consequently, $\Psi(f_\alpha)=\psi(f_\alpha|_{X_{fin}^ \rho})\to 0$. Thus $\Psi\in (X^*, \tau_{ucb})^*$. Since $\left(X, \tau \right)$ is reflexive, there exists an $x_0\in X$ such that $\Psi=J_{x_0}$. If $x_0\notin X_{fin}^\rho$, then there exists an $f\in X^*$ such that $f(x_0)\neq 0$ and $f(X_{fin}^\rho)=0$. Then $0\neq J_{x_0}(f)=f(x_0)=\Psi(f)=\psi(f|_{X_{fin}^\rho})=0$. So $x_0\in X_{fin}^\rho$. For every $f\in \left( X_{fin}^\rho, \tau|_{X_{fin}^\rho}\right)^*$, $\psi(f) = \Psi(f') = J_{x_0}(f') =  f'(x_0) = f(x_0)$, where $f'$ is a continuous linear extension of $f$ on $X$ which is possible by Corollary 4.2 in \cite{flctopology}. Hence  $\left(X_{fin}^\rho, \tau|_{X_{fin}^\rho} \right)$ is semi-reflexive.
		 
		To complete the proof it is enough to show that the canonical map  $J_\rho: \left(X_{fin}^\rho, \tau|_{X_{fin}^\rho} \right)\to \left( \left( X_{fin}^\rho\right)^*, \tau_{ucb}^\rho\right)^*$ on $X_{fin}^\rho$ is continuous. Let $M$ be a subspace of $X$ such that $X=X_{fin}^\rho\oplus M$. Suppose $D$ is any bounded set in $\left(\left( X_{fin}^\rho\right)^*, \tau_{ucb}^\rho\right)$. Consider $Z=\{f' : f \in D\}$, where $f'$ is the continuous linear extension of $f$ on $X$ which is $0$ on $M$. Then $Z$ is pointwise bounded. Since $\left(X, \tau \right)$ is reflexive, by Proposition \ref{reflexive of an elcs imply barrelledness of finest space}, $(X, \tau_F)$ is barreled. By Theorem 11.3.4 and Theorem 11.3.5, p. 384 in \cite{tvsnarici}, $Z$ is bounded in $(X^*, \tau_s)$. So $Z$ is bounded in $(X^*, \tau_{ucb})$.  Consequently, $Z^\circ$ is a neighborhood of $0$ in $(X^*, \tau_{ucb})^*$. Thus $J^{-1}\left(Z^\circ \right)$ is a neighborhood of  $0$ in $(X, \tau)$ as  $\left(X, \tau \right)$ is reflexive. Therefore $J^{-1}\left(Z^\circ \right)\cap X_{fin}^\rho$ is a neighborhood of $0$ in $\left(X_{fin}^\rho, \tau|_{X_{fin}^\rho} \right)$. Note that $J^{-1}\left( Z^\circ\right)\cap X_{fin}^\rho=J_\rho^{-1}(D^\circ)$ (if $x\in X_{fin}^\rho$, then $|f(x)|\leq 1~ \text{for every}~ f\in D\iff |g(x)|\leq 1 ~\text{for every}~ g\in Z$). Which implies that $J_\rho^{-1}(D^\circ)$  is a neighborhood of $0$ in $\left(X_{fin}^\rho, \tau|_{X_{fin}^\rho} \right)$. Therefore $J_\rho$ is continuous. Hence $\left(X_{fin}^\rho, \tau|_{X_{fin}^\rho} \right)$ is reflexive.  \\
		
		The implication (2)$\Rightarrow$(3) is obvious. 
		
		(3)$\Rightarrow$(1). Let $(3)$ hold for some continuous extended seminorm $\rho$ on $(X, \tau)$ and let $M$ be a subspace of $X$ such that $X=X_{fin}^\rho\oplus M$. Suppose $\Psi$ is a continuous linear functional on $(X^*, \tau_{ucb})$. Define  linear functionals $\Psi_1$ and $\Psi_2$ on $\left(X_{fin}^\rho, \tau|_{X_{fin}^\rho} \right)^*$ and $(M, \tau|_M)^*$, respectively,  by $\Psi_1(f)= \Psi(\hat{f})$  for $f\in\left(X_{fin}^\rho, \tau|_{X_{fin}^\rho} \right)^*$ and $\Psi_2(g)= \Psi(g')$ for $g\in (M, \tau|_M)^*$, where  \begin{alignat*}{2}
		\hat{f}(x)&=\left\{
		\begin{array}{lll}
		f(x) & \textnormal{if} & x\in X_{fin}^\rho \\
		0 & \textnormal{if} & x\in M
		\end{array}
		\right.
		&\qquad
		g'(x)&=\left\{
		\begin{array}{lll}
		0 & \textnormal{if} & x\in X_{fin}^\rho \\
		g(x) & \textnormal{if} & x\in M.
		\end{array}
		\right.	\end{alignat*}
		 It is easy to see that if nets $f_\alpha\to 0$ in $\left( \left( X_{fin}^\rho\right)^*, \tau_{ucb}^\rho\right)$ and $g_\beta\to 0$ in $\left( M^*, \tau_{ucb}^m\right)$, then $\hat{f}_\alpha\to 0$, $g_\beta'\to 0$ in $(X^*, \tau_{ucb})$, where $\tau_{ucb}^\rho$ and $\tau_{ucb}^m$ are the topologies of the uniform convergence on bounded subsets of $(X_{fin}^\rho, \tau|_{X_{fin}^\rho})$ and $(M, \tau|_M)$, respectively. Since $\Psi\in (X^*, \tau_{ucb})^*$, we have $\Psi_1\in\left( \left( X_{fin}^\rho\right)^*, \tau_{ucb}^\rho\right)^*$ and $\Psi_2\in \left( M^*, \tau_{ucb}^m\right)^*$. So there exists $x_f\in X_{fin}^\rho$ such that $\Psi_1(f)=f(x_f)$ for all $f \in (X_{fin}^\rho, \tau|_{X_{fin}^\rho})^*$ as $\left(X_{fin}^\rho, \tau|_{X_{fin}^\rho} \right)$ is semi-reflexive. By Remark \ref{reflexivity of M}, $M$ is also semi-reflexive. There exists $x_m\in M$ such that $\Psi_2(g) = g(x_m)$ for all $g \in (M, \tau|_M)^*$. Note that for every $f \in X^*$, $\Psi(f=\widehat{f|}_{X_{fin}^\rho}+f|_M')=\Psi(\widehat{f|}_{X_{fin}^\rho})+\Psi(f|_M')=\Psi_1(f|_{X_{fin}^\rho})+\Psi_2(f|_M)= f|_{X_{fin}^\rho}(x_f) + f|_M(x_m) = f(x_f+x_m) = J_{x_f+x_m}(f)$. Therefore $(X, \tau)$ is semi-reflexive.  
		 
		 Let $D$ be a bounded set in $(X^*, \tau_{ucb})$. Then it is pointwise bounded. Since both $\left( X_{fin}^\rho, \tau|_{X_{fin}^\rho}\right) $ and $(M, \tau|_M)$ are reflexive, by Theorem \ref{reflexivity of an elcs and finest space}, both the spaces $\left( X_{fin}^\rho, \tau_F|_{X_{fin}^\rho}\right) $ and $(M, \tau_F|_M)$ are barreled. By Lemma \ref{X is isomorphic to finite subspace and M} and  Theorem 11.12.4, p. 409 in \cite{tvsnarici}, $(X, \tau_F)$ is barreled. Therefore by Theorem 11.3.4, p. 384 in \cite{tvsnarici}, $D$ is equicontinuous on $(X, \tau_F)$. Consequently, $D_\circ$ is a neighborhood of $0$ in $(X, \tau_F)$. Note that $J^{-1}(D^\circ)= \{x\in X: |f(x)|\leq 1 ~\text{for}~ f\in D\}=D_\circ$. So $J^{-1}(D^\circ)$ is a neighborhood of $0$ in $(X, \tau)$. Hence $J$ is continuous and $(X, \tau)$ is reflexive. \end{proof}
	 
	 \begin{remark} A similar result holds if we replace reflexive by semi-reflexive in the statement of Theorem \ref{reflexivity of X with finite subspaces}. \end{remark}
	 
\begin{corollary}\label{reflexivity of finite subspace in an enls}
	Let $(X, \parallel\cdot\parallel)$ be an enls. Then $(X, \parallel\cdot\parallel)$ is reflexive if and only if $(X_{fin}, \parallel\cdot\parallel)$ is reflexive.\end{corollary}
\begin{corollary}\label{reflexive enls are extended banach}
	Let $(X, \parallel\cdot\parallel)$ be a reflexive enls. Then $(X, \parallel\cdot\parallel)$ is an extended Banach space. \end{corollary}
\begin{proof}
	If $(X, \parallel\cdot\parallel)$ is reflexive, then $(X_{fin}, \parallel\cdot\parallel)$ is reflexive. By classical theory for normed space, $(X_{fin}, \parallel\cdot\parallel)$ is a Banach space. Hence by Proposition 3.11 in  \cite{nwiv}, $(X, \parallel\cdot\parallel)$ is an extended Banach space.\end{proof}

\noindent Recall that for an elcs $(X, \tau)$ with the dual $X^*$, the \textit{weak topology} $\tau_w$ on $X$ is the locally convex space topology induced by the collection $\mathcal{P}_w=\{\rho_f:f\in X^*\}$ of seminorms on $X$, where $\rho_f(x)=|f(x)|$ for every $x\in X$. It is easy to prove that the weak topologies corresponding to $(X, \tau)$ and $(X, \tau_F)$ are same.  

\begin{theorem}\label{reflivity and weak compactness of closed unit ball}
	Let $(X,\parallel\cdot\parallel)$ be an extended Banach space with the flc topology $\tau_F$. Then the following assertions are equivalent:
		\begin{itemize}
			\item[(1)] $(X, \parallel\cdot\parallel)$ is reflexive;
			\item[(2)] $(X, \tau_F)$ is reflexive;
			\item[(3)] $(X_{fin}, \parallel\cdot\parallel)$ is reflexive;
			\item[(4)] the closed unit ball $B_X$ is weakly compact;
			\item[(5)] the weak topology has the Heine-Borel property.\end{itemize}\end{theorem}
\begin{proof}
(2)$\Leftrightarrow$(5). It follows from the fact that a locally convex space is semi-reflexive if and only if its weak topology has the Heine-Borel property (see, Theorem 15.2.4, p. 489 in \cite{tvsnarici}).

(4)$\Leftrightarrow$(3). Note that if $\tau_w$ is the weak topology corresponding to $(X,\parallel\cdot\parallel)$ (or $(X,\tau_F)$), then $\tau_w|_{X_{fin}}$ is the weak topology of the Banach space $(X_{fin}, \parallel\cdot\parallel)$.  Consequently, the equivalence follows from the fact that a Banach space is reflexive if and only if its closed unit ball is weakly compact (see, Exercise 15.101, p. 516 in \cite{tvsnarici}). 

The equivalences (1)$\Leftrightarrow$(2)$\Leftrightarrow$(3) follow from Corollaries \ref{reflexivity of an enls}, \ref{reflexivity of finite subspace in an enls}. \end{proof}

Recall that a normed linear space $X$ is reflexive if and only if $J(B_X)\supseteq B_{X^{**}}$, where $J$ is the canonical map on $X$, and  $B_{X}$ and $B_{X^{**}}$ are the closed unit balls in $X$  and $(X^*, \parallel\cdot\parallel_{op})^*$, respectively. We next prove an analogous result for an enls. 
 
\begin{theorem}\label{reflexive in terms of polar of closed unit ball in the dual}
	Suppose $(X,\parallel\cdot\parallel)$ is an extended normed space and $J$ is the corresponding canonical map on $X$. Then  $X$ is reflexive if and only if  $J(B_X)=\left(B_{X^*}\right)^\circ$, where $B_{X^*}$ is the closed unit ball in $(X^*, \parallel\cdot\parallel_{op})$ and $\left(B_{X^*}\right)^\circ$ is the polar of $B_{X^*}$ in $(X^*, \tau_{ucb})^*$.\end{theorem}

\begin{proof}
	It is easy to see that $J(B_X)\subseteq \left(B_{X^*}\right)^\circ$. For the reverse inclusion consider $\psi\in \left(B_{X^*}\right)^\circ$. Since $X$ is reflexive, there is an  $x_0\in X$ such that $\psi=J_{x_0}$. Therefore for any $\phi\in B_{X^*}$, we have $|\phi(x_0)|=|J_{x_0}(\phi)|=|\psi(\phi)|\leq 1$. So by Proposition 4.9 in \cite{nwiv}, $x_0\in B_X$. 
	
	Conversely, suppose $J(B_X)=\left(B_{X^*}\right)^\circ$. By Corollary \ref{reflexivity of finite subspace in an enls}, it is enough to show that $(X_{fin}, \parallel\cdot\parallel)$ is reflexive.  Let $B_{X_{fin}}$ and $B_{X_{fin}^{**}}$ be the closed unit balls in $(X_{fin}, \parallel\cdot\parallel)$ and $(X_{fin}^*, \parallel\cdot\parallel_{op})^*$, respectively. To show $(X_{fin}, \parallel\cdot\parallel)$ is reflexive, it is enough to show that $J_f(B_{X_{fin}})\supseteq B_{X_{fin}^{**}}$, where $J_f$ is the canonical map on $X_{fin}$. Suppose $\Psi\in B_{X_{fin}^{**}}$. Define a linear functional $\hat{\Psi}$ on $X^*$ by $\hat{\Psi}(f)= \Psi(f|_{X_{fin}})$ for every $f\in X^*$. It is easy to see that $|\hat{\Psi}(f)|= |\Psi(f|_{X_{fin}})|\leq \parallel \Psi \parallel_{op} \parallel f\parallel_{op} \leq 1$ for every $f\in B_{X^*}$. Then $\hat{\Psi}$ is continuous on $(X^*, \parallel\cdot\parallel_{op})$. Consequently, $\hat{\Psi}\in(X^*, \tau_{ucb})^*$. Also observe that $\hat{\Psi}\in \left(B_{X^*}\right)^\circ$. So there exists an $x\in B_{X}=B_{X_{fin}}$ such that $J_x=\hat{\Psi}$ as $J(B_X)=(B_{X^*})^\circ$. For any $\phi\in (X_{fin}, \parallel\cdot\parallel_{op})^*$, we have  $\Psi(\phi)=\hat{\Psi}(\phi')=J_x(\phi')=\phi'(x)=\phi(x)$, where $\phi'$ is any continuous linear extension of $\phi$ on $X$. Therefore $\Psi=J_f(x)\in J_f(B_{X_{fin}})$.\end{proof}

 
In the next result, we show that the reflexive property in an extended Banach space is a three-space property, that is, if $Y$ is a closed subspace of an enls $(X, \parallel\cdot\parallel)$ and any two of the spaces $(Y, \parallel\cdot\parallel)$, $(X, \parallel\cdot\parallel)$ and $\left(X/Y, \parallel\cdot\parallel_q \right)$ are reflexive, then the third one is also reflexive, where $X/Y=\{x+Y :x\in X\}$ and $\parallel x+Y\parallel_q=\inf\{\parallel x-y\parallel : y\in Y\}$.

 \begin{theorem}\label{three space property}
 	Let $(X,\parallel\cdot\parallel)$ be an enls and let $Y$ be a closed subspace of $X$. Then $(X, \parallel\cdot\parallel)$ is reflexive if and only if both $(Y, \parallel\cdot\parallel)$ and the quotient space $\left(X/Y, \parallel\cdot\parallel_q \right)$ are reflexive.\end{theorem}
 \begin{proof}
 	Suppose $(X,\parallel\cdot\parallel)$ is reflexive. Then $Y_{fin}=\{y\in Y: \parallel y\parallel<\infty \}$ is a closed subspace of a reflexive space $(X_{fin}, \parallel\cdot\parallel)$. By Theorem 15.2.7, p. 490 in \cite{tvsnarici}, $(Y_{fin}, \parallel\cdot\parallel)$ is a reflexive space. Then by Corollary \ref{reflexivity of finite subspace in an enls},  $(Y, \parallel\cdot\parallel)$ is reflexive. Note that the finite subspace $\left(X/Y\right)_{fin}$ of the quotient space $\left(X/Y, \parallel\cdot\parallel_q \right)$ is equal to $\{x+Y: x\in X_{fin}\}$ (see, Theorem 3.21 in \cite{nwiv}). It is easy to see that if $x\in X\setminus X_{fin}$ and $\parallel y\parallel=\infty$ for some $y\in Y$, then $\parallel x-y\parallel =\infty$. So  for every $x\in X_{fin}$, we have  $$\parallel x+Y\parallel_q= \inf\{\parallel x-y\parallel: y\in Y\}= \inf\{\parallel x-y\parallel: y\in Y_{fin}\}= \parallel x+Y_{fin}\parallel_q.$$ 
 	Therefore $\left(\left( X/Y\right)_{fin} , \parallel\cdot\parallel_q \right)$ is isometrically isomorphic to $\left(X_{fin}/Y_{fin}, \parallel\cdot\parallel_q \right)$. Since reflexivity in a normed linear space is a three-space property (see, p. 491 in \cite{tvsnarici}), we have  $\left(X_{fin}/Y_{fin}, \parallel\cdot\parallel_q \right)$ is reflexive. So  $\left(\left( X/Y\right)_{fin}, \parallel\cdot\parallel_q \right)$ is reflexive. By Corollary \ref{reflexivity of finite subspace in an enls},  $\left(X/Y, \parallel\cdot\parallel_q \right)$ is reflexive. 
 	
 	Conversely, suppose both $(Y, \parallel\cdot\parallel)$ and  $\left(X/Y, \parallel\cdot\parallel_q \right)$ are reflexive. By Corollary \ref{reflexivity of finite subspace in an enls}, both $(Y_{fin}, \parallel\cdot\parallel)$ and  $\left((X/Y)_{fin}, \parallel\cdot\parallel_q \right)$ are reflexive. So $(Y_{fin}, \parallel\cdot\parallel)$ and $\left(X_{fin}/Y_{fin}, \parallel\cdot\parallel_q \right)$ are reflexive. Therefore $(X_{fin}, \parallel\cdot\parallel)$ is reflexive. Hence by Corollary \ref{reflexivity of finite subspace in an enls}, $(X, \parallel\cdot\parallel)$ is reflexive. \end{proof}

In general, if $\mu$ is a finitely compatible norm for an enls $(X,\parallel\cdot\parallel)$, that is, both $\mu$ and $\parallel\cdot\parallel$ induce the same topology on $X_{fin}$, then the reflexivity of $(X, \mu)$ may not have any relation with the reflexivity of $(X, \parallel\cdot\parallel)$ (see, Examples \ref{reflexivity doesnot depend on finitely compatible norm 1}, \ref{reflexivity doesnot depend on finitely compatible norm 2}). But, there exists a finitely compatible norm $\nu$ such that $(X, \parallel\cdot\parallel)$ is reflexive whenever $(X, \nu)$ is reflexive.    

\begin{theorem} Suppose $(X,\parallel\cdot\parallel)$ is an extended normed  space with $X=X_{fin}\oplus M$. Then there exists a finitely compatible norm $\nu$ for $(X,\parallel\cdot\parallel)$ with the following property: If $(X,\nu)$ is reflexive, then $(X,\parallel\cdot\parallel)$ is reflexive.\end{theorem}

\begin{proof} Let $\phi\in X^*$ such that $\phi^{-1}(0)=X_{fin}$. Consider the norm $\nu$ on $X$ by $\nu(x)=\parallel x_f\parallel + |\phi(x_M)|$, where $x=x_f+x_M$ with $x_f\in X_{fin}$ and $x_M\in M$. Then $\nu$  is a finitely compatible norm for $(X, \parallel\cdot\parallel)$ and $\phi\in (X,\nu)^*.$ As $(X,\nu)$ is reflexive and $\phi\in (X,\nu)^*$ with  $\phi^{-1}(0)=X_{fin}$,   $(X_{fin},\parallel\cdot\parallel)$ is reflexive. By Corollary \ref{reflexivity of finite subspace in an enls}, $(X,\parallel\cdot\parallel)$ is reflexive.\end{proof}

\begin{example}\label{reflexivity doesnot depend on finitely compatible norm 1}
	Let  $c_{00}$ be the collection of all eventually zero sequences with the discrete extended norm $\parallel\cdot\parallel_{0,\infty}$. Then $(c_{00},\parallel\cdot\parallel_{0,\infty})$ is a reflexive space but the finitely compatible normed space $(c_{00},\parallel\cdot\parallel_{\infty})$ is not reflexive.\end{example} 

\begin{example}\label{reflexivity doesnot depend on finitely compatible norm 2}
	Let $l_p ~(1<p<\infty)$ be the space of all $p$-summable real sequences. Suppose $M$ is  a subspace of $l_{p}$ with $l_{p}=c_{00}\oplus M$. Define an extended norm on $l_p$ by 
	\[ \parallel x\parallel =
	\begin{cases}
	\text{$\parallel x\parallel_{p},$} &\quad\text{if $x\in c_{00} $}\\
	\text{$\infty$,} &\quad\text{otherwise. }\\
	\end{cases}\]
	Then $X_{fin}=c_{00}$ is not reflexive. Consequently, $(l_{p}, \parallel\cdot\parallel)$ is a non-reflexive extended normed linear space. Note that the classical $p$-norm $\parallel\cdot\parallel_{p}$ is a finitely compatible norm on $(X, \parallel\cdot\parallel)$. Also, the normed space $(l_p,\parallel\cdot\parallel_{p})$ is a reflexive space.\end{example}

\section{Applications to Function spaces}
Let $(X, d)$ be a metric space and let $C(X)$ be the set of all real-valued continuous functions on $(X,d)$. By a \textit{bornology} $\mathcal{B}$ on $(X, d)$, we mean a collection of nonempty subsets of $X$ that covers $X$ and is closed under finite union and taking subsets of its members. A subfamily $\mathcal{B}_0$ of $\mathcal{B}$ is a \textit{base} for $\mathcal{B}$ if it is cofinal in $\mathcal{B}$ under the set inclusion. In addition, if every member of $\mathcal{B}_0$ is closed in $(X,d)$, then we say $\mathcal{B}$ has a \textit{closed base}. For more details about metric bornologies, we refer to \cite{bala}.

In this section, we study the reflexivity of the function spaces $(C(X), \tau_{\mathcal{B}}^s)$ and $(C(X), \tau_{\mathcal{B}})$, where $\tau_{\mathcal{B}}^s (\tau_{\mathcal{B}})$ is the topology of strong uniform convergence (uniform convergence) on the elements of $\mathcal{B}$. We first define these topologies. 
%
%
%
%
 \begin{definition}{\normalfont (\cite{Suc})} \normalfont Let $\mathcal{B}$ be a bornology on a metric space $(X, d)$. Then \textit{the topology $\tau_{\mathcal{B}}$ of  uniform convergence on $\mathcal{B}$} is determined by a uniformity on $C(X)$ having base consisting of sets of the form  $$[B, \epsilon]=\left\lbrace (f, g) : \forall x\in B,~ |f(x)-g(x)|<\epsilon \right\rbrace~ (B\in\mathcal{B},~ \epsilon>0). $$  \end{definition}

 \begin{definition}{\normalfont (\cite{Suc})} \normalfont Let $\mathcal{B}$ be a bornology on a metric space $(X, d)$. Then \textit{the topology $\tau^{s}_{\mathcal{B}}$ of strong uniform convergence on $\mathcal{B}$} is determined by a uniformity on $C(X)$ having base consisting of sets of the form  $$[B, \epsilon]^s=\left\lbrace (f, g) : \exists~ \delta>0 ~\forall x\in B^\delta,~ |f(x)-g(x)|<\epsilon \right\rbrace~ (B\in\mathcal{B},~ \epsilon>0),$$ where for $B\subseteq X$, $B^\delta=\displaystyle{\bigcup_{y\in B}}\{x\in X: d(x, y)<\delta\}$.  	 \end{definition}
 
Suppose $\mathcal{B}$ is a bornology on a metric space $(X, d)$. Then the topology $\tau_{\mathcal{B}}$ on $C(X)$ is induced by the collection $\mathcal{P}=\{\rho_{B}: B\in\mathcal{B}\}$ of extended seminorms, where $\rho_{B}(f)=\sup_{x\in B}|f(x)|$ for $f\in C(X)$. Similarly, the topology $\tau_{\mathcal{B}}^s$ on $C(X)$ is induced by the collection $\mathcal{P}=\{\rho_{B}^s: B\in\mathcal{B}\}$ of extended seminorms, where $\rho_B^s(f) =\inf_{\delta>0}\left\lbrace \sup_{x\in B^\delta}|f(x)|\right\rbrace$ for $f\in C(X)$. Hence both the function spaces $(C(X), \tau_{\mathcal{B}})$ and $(C(X), \tau_{\mathcal{B}}^s)$ are actually extended locally convex spaces. For more details related to $\tau_{\mathcal{B}}$ and $\tau_{\mathcal{B}}^s$, we refer to \cite{Suc, ATasucob, SWSUCB}.  
  
%
%

Recall that if $(X, \tau)$ is a locally convex space with the dual $X^*$, then the \textit{Mackey topology} $\tau_M$ is a locally convex space topology on $X$ whose neighborhood base at $0$ is given by $$\mathcal{B}_M=\{B_\circ: B \text{ is an absolutely convex and weak$^*$ compact subset of } X^*\}.$$
We say $(X, \tau)$ is a \textit{Mackey space} if $\tau=\tau_M$. Note that $\tau_M$ is the largest locally convex topology on $X$ such that $(X,\tau_M)^*=X^*$. For more details about Mackey topology, we refer to \cite{tvsnarici, lcsosborne}.    

\begin{proposition}\label{mackey topology for weak* topology} Suppose $\tau_{w^*}$ and $\tau_w$ are the weak$^*$ and weak topologies for an elcs $(X, \tau)$, respectively. Then $$\mathcal{B}=\{B^\circ: B \text{ is absolutely convex and compact in } (X, \tau_w)\}$$ is a neighborhood base at $0$ in $(X^*, \tau_M)$, where $\tau_M$ is the Mackey topology for the locally convex space $(X^*, \tau_{w^*})$. 	     \end{proposition}
\begin{proof} Let $D$ be an absolutely convex and weak$^*$ compact subset of  $(X^*, \tau_{w^*})^*$. Then $D=\{J_x:x\in A\}$ for some $A\subseteq X$ as $(X^*, \tau_{w^*})^*=\{J_x:x\in X\}$. Since $D$ is absolutely convex, $A$ is absolutely convex. Note that a net $(x_\lambda)$ in $X$ converges weakly to $x$ if and only if $f(x_\lambda)\to f(x)$ for every $f\in X^*$ if and only if $J_{x_\lambda}(f)\to J_x(f)$ for all $f\in X^*$. Therefore $A$ is weakly compact. It is easy to prove that $A^\circ=D_\circ$. Which completes the proof.   
\end{proof}

\begin{theorem}\label{reflexive iff every bdd set is weakly cpt} Suppose $(X, \tau)$ is an elcs. Then $(X, \tau)$ is semi-reflexive if and only if every bounded subsets of $(X, \tau)$ is relatively weakly compact.\end{theorem}

\begin{proof} Suppose $(X, \tau)$ is semi-reflexive. Then $(X^*, \tau_{ucb})^*=(X^*, \tau_{w^*})^*=\{J_x:x\in X\}$. Therefore $\tau_{ucb}\subseteq \tau_M$, where $\tau_M$ is the Mackey topology for $(X^*, \tau_{w^*})$. Now, let $A$ be any bounded set in $(X, \tau)$. Then $A^\circ$ is a neighborhood of $0$ in $(X^*, \tau_{ucb})$. Consequently, $A^\circ$ is a neighborhood of $0$ in $(X^*, \tau_M)$. By Proposition \ref{mackey topology for weak* topology}, there exists an absolutely convex and weakly compact subset $B$ of $X$ such that $B^\circ\subseteq A^\circ$. Therefore $A\subseteq (A^\circ)_\circ\subseteq (B^\circ)_\circ$. By applying Bipolar Theorem on $B$ in $(X, \tau_w)$, we obtain $(B^\circ)_\circ=B$. Hence $A$ is relatively weakly compact.

Conversely, suppose every bounded set is weakly compact in $(X, \tau)$. Then $\tau_{ucb}\subseteq \tau_M$. Therefore $(X^*, \tau_{ucb})^*=(X^*, \tau_M)^*=(X^*, \tau_{w^*})^*=\{J_x:x\in X\}$. Hence $(X, \tau)$ is semi-reflexive.  	         
\end{proof}

\begin{theorem}\label{reflexivity of tauB implies K is subset of B} Let $\mathcal{B}$ be a bornology with a closed base on a metric space $(X, d)$. If $(C(X), \tau_{\mathcal{B}}^s)$ $(\text{or } (C(X), \tau_{\mathcal{B}}))$ is reflexive, then $\mathcal{K}\subseteq \mathcal{B}$.	\end{theorem}
\begin{proof} It follows from Theorems 3.15, 4.3 in \cite{barreledelcs} and  Proposition \ref{reflexive of an elcs imply barrelledness of finest space}.
\end{proof}

\begin{theorem}\label{semi-reflexivity on tauBs} Let $\mathcal{B}$ be a bornology with a closed base on a metric space $(X, d)$. If $(C(X), \tau_{\mathcal{B}}^s)$ $(\text{or } (C(X), \tau_{\mathcal{B}}))$ is reflexive, then $X$ is a discrete space.	\end{theorem}

\begin{proof} Suppose $K$ is any compact set in $(X, d)$. Then by Theorem \ref{reflexivity of tauB implies K is subset of B}, $K\in \mathcal{B}$.  We show that $K$ is a finite set. Let $(x_n)$ be any sequence in $K$ converging to $x_0$. For every $n\in\mathbb{N}$, there exists an $f_n\in C(X, [0, 1])$ such that $f_n(x_j)=0$ for $1\leq j\leq n$ and $f_n(x_0)=1$. Consider $A=\{f_n: n\in\mathbb{N}\}$. Clearly, $A$ is bounded in $(C(X), \tau_{\mathcal{B}}^s)$. By Theorem \ref{reflexive iff every bdd set is weakly cpt}, $A$ is relatively weakly compact in $(C(X), \tau_{\mathcal{B}}^s)$. Therefore there exists a cluster point $f$ of $A$. It is easy to prove that the topology of pointwise convergence is coarser than the weak topology for $(C(X), \tau_{\mathcal{B}}^s)$ as the  maps $\Psi_x:\left( C(X), \tau_{\mathcal{B}}^s\right) \to \mathbb{R}$ defined by $\Psi_x(f)=f(x)$ is continuous for every $x\in X$. For every $m\in \mathbb{N}$ and $\epsilon >0$, there exists an $n>m$ such that $|f(x_m)-f_n(x_m)|<\epsilon$. Therefore $f(x_m)=0$ for every $m\in \mathbb{N}$. Consequently, $f(x_m)\to f(x_0)=0$. But, there also exists an $n\in \mathbb{N}$ such that $|f(x_0)-f_n(x_0)|<\frac{1}{2}$. We arrive at a contradiction. Hence $X$ is a discrete space.       
\end{proof}

%

Suppose $(X,d)$ is a metric space. Then the topology $\tau_u$ of uniform convergence on $C(X)$ is induced by the extended norm $\parallel f\parallel_\infty=\sup_{x\in X}|f(x)|$. It is known that for a compact space $(X,d)$, the normed space $(C(X), \parallel\cdot\parallel_\infty)$ is reflexive if and only if $X$ is finite (Example 15.5.2, p. 502 in \cite{tvsnarici}).  We now prove a similar result without assuming $(X,d)$ to be compact. The next theorem also shows that the converse of Theorem \ref{semi-reflexivity on tauBs} may not be true. 

\begin{theorem} Suppose $(X,d)$ is a metric space. Then the uniform space $(C(X), \parallel\cdot\parallel_\infty)$ is reflexive if and only if $X$ is finite.\end{theorem}

\begin{proof}If $X$ is finite, then  $C(X)$ is finite dimensional. Therefore $(C(X), \parallel\cdot\parallel_\infty)$ is reflexive. Conversely, suppose $(C(X), \parallel\cdot\parallel_\infty)$ is reflexive and $X$ is infinite. Then by Example 15.5.2, p. 502 in \cite{tvsnarici}, $X$ cannot be compact. So there exists a countable, closed and discrete subset $T=\{t_n:n\in\mathbb{N}\}$ of $X$. If $Y=\{f\in C(X): f=0~ \text{on}~ T\}$, then $Y$ is a closed subspace of $C(X)$. By  Theorem \ref{three space property}, $C(X)/Y$ is reflexive. Now, we  show that the space $l_\infty$ of all bounded real sequences is isometrically isomorphic to a subspace of $C(X)/Y$. Let $z=(z_n)\in l_\infty$ and $m=\displaystyle{\inf_{n\in\mathbb{N}}}  z_n$ and $M=\displaystyle{\sup_{n\in\mathbb{N}}}  z_n$. Since $T$ is discrete and closed, by Tietze extension theorem, there exists a $f_z\in C(X)$ such that $f_z(t_n)=z_n$ for $n\in\mathbb{N}$. Define $F_z(x) = \max\{m, \min\{M, f_z(x)\}\}$ for $x\in X$. Then $F_x\in C(X)$ with $F_x(t_n)=x_n$ for $n\in\mathbb{N}$. Consider $\Psi: l_\infty\to C(X)/Y$ by $\Psi(z)=F_z+Y$. Then $\Psi$ is linear as $F_{\alpha x +y} = \alpha F_x+F_y ~\text{on}~ Y$ for $x, y \in l_\infty$ and $\alpha\in\mathbb{R}$. Note that if $z=(z_n)\in l_\infty$ and $f\in Y$, then $\parallel F_z+Y\parallel \leq \parallel F_z\parallel_\infty\leq \max \{|m|, |M|\} = \parallel z\parallel_\infty$ and $\parallel z\parallel_\infty = \displaystyle{\sup_{n\in\mathbb{N}}} |z_n| = \displaystyle{\sup_{n\in\mathbb{N}}}|F_z(t_n)| = \displaystyle{\sup_{n\in\mathbb{N}}}|F_z(t_n)-f(t_n)| \leq \parallel F_z-f\parallel_\infty$. Therefore $\Psi$ is an isometry. By Theorem \ref{three space property}, $\Psi(l_\infty)$ is reflexive. Consequently, by Exercise 3.61, p. 97 in \cite{faaidg}, $l_\infty$ is reflexive. Which is not true.     \end{proof}

\bibliographystyle{plain}
\bibliography{reference_file}		

\end{document}